\documentclass[11pt,a4paper,lualatex]{amsart} 

\usepackage[marginparwidth=0pt,margin=20truemm]{geometry} 

\usepackage{mypackage} 
\usepackage{mycommand} 

\usepackage{physics} 

\DeclareMathOperator{\Resultant}{Res}
\renewcommand{\Res}{\Resultant}

\newcommand{\cyc}{\Phi^{\mathrm{cyc}}} 

\newcommand{\ab}{\mathrm{ab}}

\usepackage{ctable} 

\usepackage[pdfencoding=auto,hypertexnames=false]{hyperref}
\hypersetup{colorlinks=false}

\numberwithin{equation}{section} 
\usepackage{mytheoremeng} 

\usepackage{tikz}
\usetikzlibrary{arrows.meta, decorations, decorations.markings}

\begin{document}

\title[Irreducibility of $\Delta_{3k,3}$]{Irreducibility of polynomials defining parabolic parameters of period $3$}

\author[J. Koizumi]{Junnosuke Koizumi}
\address{RIKEN iTHEMS, Wako, Saitama 351-0198, Japan}
\email{junnosuke.koizumi@riken.jp}

\author[Y. Murakami]{Yuya Murakami}
\address{Faculty of Mathematics, Kyushu University, 744, Motooka, Nishi-ku, Fukuoka, 819-0395, JAPAN}
\email{murakami.yuya.896@m.kyushu-u.ac.jp}

\author[K. Sano]{Kaoru Sano}
\address{NTT Institute for Fundamental Mathematics, NTT Communication Science Laboratories, NTT, Inc., 2-4, Hikaridai, Seika-cho, Soraku-gun, Kyoto 619-0237, Japan}
\email{kaoru.sano@ntt.com}

\author[K. Takehira]{Kohei Takehira}
\address{Graduate School of Science, Tohoku University, 
	6-3, Aoba, Aramaki-aza, Aoba-ku, 
	Sendai, 980-8578, Japan
        (when the research was conducted) \\
        Present address: NTT DATA Mathematical Systems Inc., Tokyo, Japan}
\email{kohei.takehira.p5@dc.tohoku.ac.jp (when the research was conducted) \\
    Present e-mail address: takehira@msi.co.jp}

\date{\today}


\begin{abstract}
    Morton and Vivaldi defined the polynomials whose roots are parabolic parameters for a one-parameter family of polynomial maps. We call these polynomials delta factors.
    They conjectured that delta factors are irreducible for the family $z\mapsto z^2+c$.
    One can easily show the irreducibility for periods $1$ and $2$ by reducing it to the irreducibility of cyclotomic polynomials.
    However, for periods $3$ and beyond, this becomes a challenging problem.
    This paper proves the irreducibility of delta factors for the period $3$ and demonstrates the existence of infinitely many irreducible delta factors for periods greater than $3$.
\end{abstract}

\maketitle


\tableofcontents



\section{Introduction} \label{sec: intro}


This paper discusses the irreducibility of polynomials that appear in connection with discrete dynamical systems of a one-parameter family of polynomial maps.

Consider the polynomial $ f_c(z) = z^2 + c \in \Z [c][z] $ parametrized by a parameter $c$.
The periodic points of the dynamical system $ z \mapsto f_c(z) $ have long been a subject of interest.
For a positive integer $m$, the $ m $-th \textit{dynatomic polynomial} $\Phi_m^{*}$ is defined by
\[
    \Phi_m^{*}(z,c) \coloneqq \prod_{k \mid m} (f_c^{\circ k}(z) - z)^{\mu(m/k)},
\]
where $ f_c^{\circ k} $ is recursively defined by $f_c^{\circ 0}(z)=z$ and $f_c^{\circ k}=f_c\circ f_c^{\circ (k-1)}$ and $ \mu $ is the M\"{o}bius function.
For general properties of the dynatomic polynomials, see \cite{Silbook00}.

One important quantity associated with periodic points is the multiplier.
For a root $ \alpha \in \overline{\Q(c)} $ of $ \Phi_m^{*}({-},c) $, the \textit{multiplier} at $ \alpha $ is defined by $ \omega_m(\alpha) = (f_c^{\circ m})'(\alpha). $
Since the chain rule shows the equality
\[
    (f_c^{\circ m})'(\alpha) = \prod_{k=0}^{m-1} f_c'(f_c^{\circ k}(\alpha)),
\]
the multiplier is constant on the orbit, that is, the equality $ \omega_m(\alpha) = \omega_m(f_c(\alpha)) $ holds.
The $ m $-th \textit{multiplier polynomial} $ \delta_m $ is defined to be the monic polynomial satisfying
\[
    \delta_m(x,c)^m
    \coloneqq \Res_z(\Phi_m^{*}(z,c), x - (f_c^{\circ m})'(z))
    = \prod_{\alpha\colon \Phi_m^{*}(\alpha,c)=0}(x - \omega_m(\alpha)),
\]
where $ \Res_z $ is the resultant with respect to the variable $ z $.
The multiplier polynomial $ \delta_m $ is indeed a polynomial with integer coefficients (cf \cite{Vivaldi-Hatjispyros}). 
Here, we remark that Huguin~\cite{Hug21} and Murakami--Sano--Takehira~\cite{MST24} independently proved that $ \delta_m $ is in $ \Z[4c,x] $ and is monic in $4c$.
\cite{MST24} discussed generalizations of this fact for other one-parameter families.

An $m$-periodic point $ \alpha $ whose multiplier $ \omega_m(\alpha) $ is a root of unity is said to be \textit{parabolic}.
We say that a parameter $ \gamma $ is \textit{parabolic} if $ f_\gamma $ has a parabolic periodic point.
Parabolic parameters are precisely the roots of the polynomials $ \Delta_{n,m}\;(m,n>0,\;m\mid n)$ defined in the following way.
Let $\cyc_k(x)$ denote the $k$-th cyclotomic polynomial.
For positive integers $m,n$ which satisfy $m \mid n$ and $m<n$, the \textit{delta factor} $\Delta_{n,m}$ is defined by the equation
\begin{equation}\label{eq: Delta}
    \Delta_{n, m}(c) \coloneqq \Res_x(\cyc_{n/m}(x), \delta_m(x, c)).
\end{equation}
In the case where $ m = n $, the delta factor $ \Delta_{n, n}(c) $ is defined by
\begin{equation}\label{eq: Delta_nn}
\delta_n(1, c) = \Delta_{n, n}(c) \prod_{m \mid n, m\neq n} \Delta_{n, m}(c).
\end{equation}
The list of $ \Delta_{n,m}$ can be found in \cite[Section 3, Table 1]{Morton-Vivaldi} or \cite[Appendix B]{MST24}.
It is known that
\begin{align}
    \deg_c \Delta_{n,m} = 
    \begin{cases}
        \nu(m) \varphi(n/m) & \text{ if } m \mid n, n \neq m, \\
        \nu(n) - \sum_{k|n, k\neq n} \nu(k)\varphi(n/k) & \text{ if } n = m,
    \end{cases} 
\end{align}
where $ \varphi(n) \coloneqq \abs{(\Z/n\Z)^\times} $ is the Euler's totient function and $ \nu(n) \coloneqq \sum_{k \mid n} 2^{k-1} \mu(n/k) $; see \cite[Corollary 3.3]{Morton-Vivaldi}.
Since the parabolic parameters are algebraic numbers, looking at their number-theoretic properties is natural.
For example, Buff--Koch determined all totally real parabolic parameters in \cite{BK22}, and Murakami--Sano--Takehira determined all quadratic parabolic parameters in \cite{MST24}.

Morton--Vivaldi conjectured the irreducibility of $\Delta_{n, m}$ over $ \Q $ in \cite{Morton-Vivaldi}.
This conjecture is also cited as an open problem in \cite[Exercise 4.12 (e)**]{Silbook00}.

\begin{conj}[The irreducibility conjecture \cite{Morton-Vivaldi}]\label{conj: irreducibility}
    For a family of the polynomial maps $f_c(z) = z^2+c$, the polynomials $\Delta_{n,m}$ for positive integers $m, n$ with $m \mid n$ are all irreducible over $\Q$.
\end{conj}

One can easily show the irreducibility of $\Delta_{n,m}$ for $m = 1, 2$ by reducing it to the irreducibility of cyclotomic polynomials; see \cite[Corollary 3.7]{Morton-Vivaldi}.
However, for $m = 3$ and beyond, this becomes a challenging problem.
\cref{fig: parabolic_parameters_of_period_3} shows the configuration of roots of $ 
\Delta_{3k,3}(c) (1\leq k \leq 79) $.
It is known that all parabolic parameters (red points in \cref{fig: parabolic_parameters_of_period_3}) lie on the boundary of the Mandelbrot set (the gray area of \cref{fig: parabolic_parameters_of_period_3}).

\begin{figure}[h]
    \includegraphics[width=15cm]{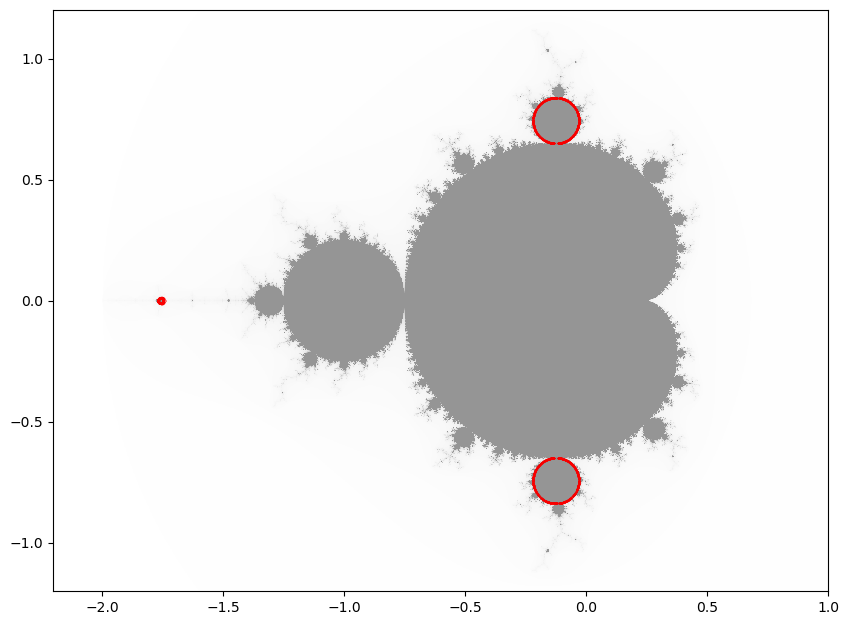}
    \caption{Parabolic parameters of period $3$}
    \label{fig: parabolic_parameters_of_period_3}
\end{figure}

This paper aims to give two partial solutions to \cref{conj: irreducibility}.
Our first main theorem is the irreducibility of the delta factors defining parabolic parameters of period $ 3 $.

\begin{thm}\label{thm: irreducibility for period 3}
    For any positive integer $k$ and $\ell$ with $(k,\ell)=1$, the polynomial $\Delta_{3k, 3}$ is irreducible over $ \Q(\zeta_{\ell}) $, where $\zeta_{\ell}$ is an $\ell$-th root of unity.
\end{thm}

\begin{rem}
    For positive integers $k,\ell, m$, if $(k,\ell)\neq 1$, the delta factor $\Delta_{km,m}$ is never irreducible over $\Q(\zeta_\ell)$ by definition.
\end{rem}

Recall that $\Delta_{n, m}(c)$ is an element of $\mathbb{Z}[4c]$.
Our second main theorem says that the irreducibility of $\widetilde{\Delta}_{mk,m}(C):=\Delta_{mk, m}(C/4)\in \mathbb{Z}[C]$
over $ \bbF_p$ implies the irreducibility of $\Delta_{mkp^e, m}$ 
over $ \bbQ $.

\begin{thm} \label{thm: irreducibility inherited}
	Let $ m \ge 1 $ and $ k \ge 2 $ be integers and $ p $ a prime number with $ p \nmid k $.
	If $ \widetilde{\Delta}_{mk,m}(C) $ is irreducible over $ \bbF_p $, then $ \Delta_{mk p^e, m} $ is irreducible over $ \Q $ for any $ e \ge 1 $.
\end{thm}

This theorem proves the existence of infinitely many irreducible delta factors for periods greater than $3$.
For example, it follows that the polynomials $\Delta_{8\cdot 11^e, 4}$ are irreducible for all $e\geq 1$.
See \cref{tab: irreducibility inherited} for more examples.

\subsection*{Organization of this paper}

This paper is organized as follows.
\cref{sec: preliminaries} is devoted to preparing some basic facts on complex dynamics (\cref{subsec: Dynamics}) and a plane curve defined by the multiplier polynomial (\cref{subsec: Multiplier curves}). In this section, a polynomial $ \Gamma_k $ is introduced, and we prove that $ \Delta_{3k,3} $ is irreducible if and only if $ \Gamma_k $ is irreducible.
In \cref{sec: distribution}, we estimate the configuration of roots of $ \Gamma_k $. More precisely, we prove some inequalities on $ \Re(\alpha) $ and $ |\alpha| $ for a root $ \alpha $ of $ \Gamma_k $.
A proof of \cref{thm: irreducibility for period 3}, our first main theorem, is given in \cref{sec: irreducibility}.
In \cref{sec: higher periods}, we prove \cref{thm: irreducibility inherited}, which is our second main theorem.



\section*{Acknowledgement}


The authors would like to thank Masanobu Kaneko, Yasuo Ohno, and Nobuo Tsuzuki for giving many helpful comments.
The first author is supported by JSPS KAKENHI Grant Number JP22J20698.
The second author is supported by JSPS KAKENHI Grant Number JP23KJ1675.
The fourth author is financially supported by JSPS KAKENHI Grant Number JP22J20227 and the WISE Program for AI Electronics, Tohoku University.


\section{Preliminaries} \label{sec: preliminaries}



\subsection{Basic facts on dynamics and delta factors}
\label{subsec: Dynamics}


This subsection recalls some basic facts on holomorphic dynamics and delta factors.
First, we review a well-known relationship between parabolic periodic points and critical points.

\begin{thm}[{cf \cite[Theorem 10.15]{Milnorbook06}}] \label{thm: parabolic_attracting_critical}
    Let $f$ be a rational map on $\bbP^1$ of degree $d\geq 2$.
    Let $\alpha \in \bbP^1(\mathbb{C})$ be a parabolic $m$-periodic point of $f$.
    Then, there is a critical point $\beta$ of $f$ satisfying
    \[
        \lim_{k\to\infty} f^{\circ (mk+i)}(\beta)= \alpha
    \]
    for some $0 \leq i \leq m-1$.
\end{thm}

Since a critical point of $ f_c(z) = z^2 + c $ is either $ 0 $ or $ \infty $, we can easily deduce the following corollary.

\begin{cor}\label{cor: uniqueness of parabolic orbit}
    For a complex number $ c \in \mathbb{C} $, the polynomial $ f_c(z) = z^2 + c $ has at most one parabolic periodic orbit. 
\end{cor}

Recall that the multiplier polynomial $\delta_m$ is an element of $\mathbb{Z}[x, 4c]$ and is monic in $4c$.
We define a polynomial $\widetilde{\delta}_m\in \mathbb{Z}[x, C]$ by $\widetilde{\delta}_m(x, C) = \delta_m(x, C/4)$.
For $m, n>0$ with $m\mid n$, we define a polynomial $\widetilde{\Delta}_{n, m}$ by $\widetilde{\Delta}_{n, m}(C) = \Delta_{n,m}(C/4)$.
It follows that $\widetilde{\Delta}_{n, m}$ is a monic polynomial in $\mathbb{Z}[C]$.

\begin{prop}
    For any $m, n > 0$ with $m \mid n$, the polynomial $\widetilde{\Delta}_{n, m}$ is separable.
\end{prop}

\begin{proof}
    By \cite[Proposition 3.2]{Morton-Vivaldi}, the polynomial $\widetilde{\delta}_n(1, C)$ is separable.
    Since we have
    $$
    \widetilde{\delta}_n(1, C) = \prod_{m \mid n} \widetilde{\Delta}_{n, m}(C),
    $$
    the polynomials $\widetilde{\Delta}_{n, m}$ are also separable.
\end{proof}


\subsection{Multiplier curves}
\label{subsec: Multiplier curves}

Throughout this section, for a polynomial $F$ of complex coefficients, we write $Z(F)$ for the set of complex roots of $F$ without counting multiplicity.

To study parabolic parameters, we introduce the following geometric object.
\begin{dfn}
    We define the \emph{multiplier curve} $X_m$ to be the curve in $\mathbb{A}^2_{\mathbb{Q}}$ defined by $\widetilde{\delta}_m(x,C)=0$.
\end{dfn}

It is known that the curve $X_m$ is geometrically integral \cite[Corollary 1]{Morton96}.
Let $k\geq 2$ be an integer.
By definition, the roots of $\widetilde{\Delta}_{mk,m}$ are the numbers $\gamma\in \overline{\mathbb{Q}}$ such that $\widetilde{\delta}_m(\zeta,\gamma)=0$ for some primitive $k$-th root of unity $\zeta$.
Since we have the uniqueness of parabolic periodic orbits by \cref{cor: uniqueness of parabolic orbit}, such $\zeta$ is uniquely determined by $\gamma$.
In other words, we have a bijection
\begin{align}\label{eq: bijection between subsets of Xm and Z(Delta)}
    \{(\zeta,\gamma)\in X_m(\overline{\mathbb{Q}})\mid \cyc_k(\zeta)=0\}\xlongrightarrow{\sim}Z(\widetilde{\Delta}_{mk,m});\quad (\zeta,\gamma)\mapsto \gamma.
\end{align}
which is Galois-equivariant, i.e., compatible with $\Gal(\overline{\Q}/\Q)$-action.
In particular, the polynomial $\widetilde{\Delta}_{mk,m}$ is irreducible over $\mathbb{Q}$ if and only if $\Gal(\overline{\mathbb{Q}}/\mathbb{Q})$ acts transitively on the set $\{(\zeta,\gamma)\in X_m(\overline{\mathbb{Q}})\mid \cyc_k(\zeta)=0\}$.

\begin{lem}\label{lem: parabolic parameter involves cyclotomic field}
        Let $k\geq 2$ be an integer.
        For any $\gamma\in Z(\widetilde{\Delta}_{mk,m})$, the field $\Q(\gamma)$ contains the cyclotomic field $\Q(\zeta_k)$.
        In particular, the degree of an irreducible factor of $\widetilde{\Delta}_{mk,m}$ over $\mathbb{Q}$ is divisible by $\varphi(k)$.
\end{lem}

\begin{proof}
        By the existence of the bijection \cref{eq: bijection between subsets of Xm and Z(Delta)}, there is a unique primitive $k$-th root of unity $\zeta$ such that $(\zeta,\gamma)\in X_m(\overline{\mathbb{Q}})$.
        If $\sigma\in \Gal(\overline{\mathbb{Q}}/\mathbb{Q}(\gamma))$, then $(\sigma(\zeta),\sigma(\gamma))=(\sigma(\zeta),\gamma)$ is also an element of $X_m(\overline{\mathbb{Q}})$.
        By the uniqueness of $\zeta$, we have $\zeta=\sigma(\zeta)$.
        This implies $\mathbb{Q}(\zeta)\subset \mathbb{Q}(\gamma)$ by Galois theory.
\end{proof}

Now, we focus on the case of $m=3$.
The multiplier curve $X_3$ is given by
\[
    \widetilde{\delta}_3(x, C) = x^2 - (2C + 16)x + (C^3 + 8C^2 + 16C + 64) = 0.
\]
It is rational since this is a cubic curve with a nodal singular point $(x, C)=(8, 0)$.
\begin{lem}\label{lem: X3_parametrization}
    The normalization of $X_3$ is given by
    \[
        \psi\colon \mathbb{A}^1_{\mathbb{Q}}\longrightarrow X_3;\quad t\mapsto (t^3-t^2+7t+1, -t^2-7).
    \]
\end{lem}

\begin{proof}
    A direct computation shows that $\psi$ is well-defined.
    For each $(x, C)\in X_3(\overline{\mathbb{Q}})$, we have
    \[
        \#\psi^{-1}(x, C) =
        \begin{cases}
            1 & \text{if }(x, C)\neq(8, 0),\\
            2 & \text{if }(x, C)=(8, 0).
        \end{cases}
    \]
    This observation shows that $\psi$ is finite and birational.
    Since $\mathbb{A}^1_\mathbb{Q}$ is normal, this map gives the normalization of $X_3$.
\end{proof}

\begin{rem}
    The parametrization of $X_3$ given in \cref{lem: X3_parametrization} is equivalent to the one provided in \cite{GF95} via the coordinate change $t\mapsto - 2\Omega - 1$.
\end{rem}

\begin{dfn}
    Let $k\geq 2$ be an integer.
    We define a monic polynomial $\Gamma_k\in \mathbb{Q}[t]$ by
    \[
        \Gamma_k(t) = \cyc_k(t^3 - t^2 + 7t + 1).
    \]
    Note that $\deg \Gamma_k = 3\varphi(k) = \deg \widetilde{\Delta}_{3k,3}$ and $\Gamma_k(0) = \cyc_k(1) \neq 0$.
\end{dfn}

\begin{lem}\label{lem: gamma_vs_delta}
    Let $k\geq 2$ be an integer.
    We have a Galois-equivariant bijection
    \begin{align}\label{eq: bijection between roots of gamma and delta}
        Z(\Gamma_k)\xlongrightarrow{\sim} Z(\widetilde{\Delta}_{3k,3});\quad \alpha \mapsto -\alpha^2 - 7.
    \end{align}
    In particular, the polynomial $\widetilde{\Delta}_{3k,3}$ is irreducible over $\mathbb{Q}$ if and only if $\Gamma_k$ is irreducible over $\mathbb{Q}$.
\end{lem}

\begin{proof}
    By \cref{lem: X3_parametrization}, we have a Galois-equivariant bijection
    \[
        Z(\Gamma_k)\xlongrightarrow{\sim}\{(\zeta,\gamma)\in X_m(\overline{\mathbb{Q}})\mid \cyc_k(\zeta)=0\}; \quad \alpha\mapsto (\alpha^3-\alpha^2+7\alpha+1, -\alpha^2-7).
    \]
    We get the desired bijection by combining this with \cref{eq: bijection between subsets of Xm and Z(Delta)}.
    Since the irreducibility of a separated polynomial over $\Q$ is equivalent to the transitivity of the $\Gal(\overline{\Q}/\Q)$-action on the set of roots, the latter statement holds.
\end{proof}

\begin{rem}
    The genus of the multiplier curve $X_m$ is computed in \cite[Theorem C]{Morton96}.
    In particular, the curve $X_4$ is also of genus $0$.
    An explicit parametrization is given in \cite{GF95}: the normalization of $X_4$ is given by
    \[
        \mathbb{A}^1_{\mathbb{Q}}\setminus\{0\}\longrightarrow X_4;\quad t\mapsto (-t^4 - 2t^3 - 4t^2 - 6t + 5 + 8t^{-1} + 16t^{-2},-t^2 - 3 - 4t^{-1}).
    \]
\end{rem}


\section{Distribution of roots of $\Gamma_k$} \label{sec: distribution}

Let $k\geq 2$ be an integer.
In this section, we set $g(t) = t^3 - t^2 + 7t + 1$, so that $\Gamma_k(t)=\cyc_k(g(t))$.
This section aims to give some estimates on the distribution of the roots of $\Gamma_k$.

\begin{dfn}
    We define subsets $\bbC_L, \bbC_R, A, B$ of $\bbC$ as follows:
    \begin{align}
        \bbC_L &= \left\{z\in \bbC\mid \Re(z)\leq 0\right\},\\
        \bbC_R &= \left\{z\in \bbC\mid \Re(z)> 0\right\},\\
        A &= \left\{z\in \bbC_L\mid |g(z)| = 1\right\}, \text{ and}\\
        B &= \left\{z\in \bbC_R\mid |g(z)| = 1\right\}.
    \end{align}
    We also set $s_1=0.275$, $s_2=2.75$, $s_3=0.495$, and $s_4=0.64$.
\end{dfn}

Since the roots of $\Gamma_k$ are contained in $A\cup B$, we want to study the location of the sets $A$ and $B$.
We summarize the estimates in this section in \cref{fig: evaluation_of_roots}. 

\begin{figure}[htb] 
    \includegraphics[width=8cm]{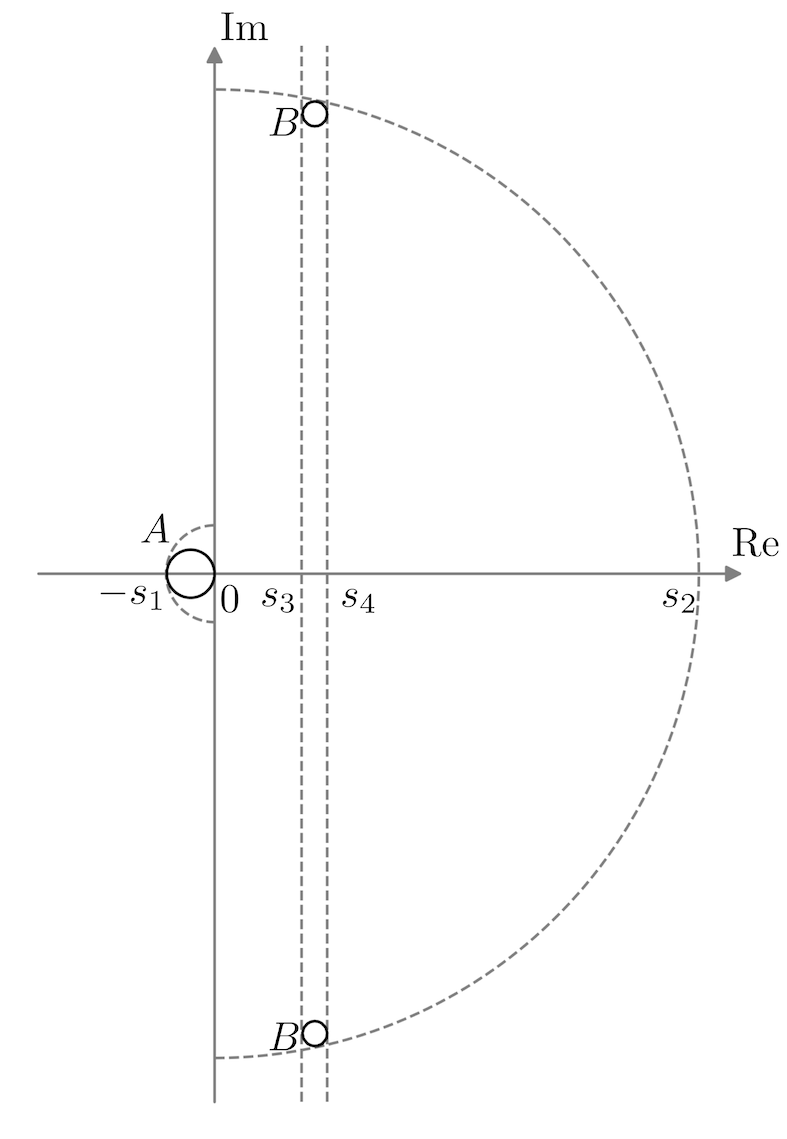}
    \caption{Location of the sets $A$ and $B$}
    \label{fig: evaluation_of_roots}
\end{figure}

We note that $A$ and $B$ are the sets of solutions of the equation $|g(z)|^2 - 1=0$.
It is easy to see that $|g(x + yi)|^2 - 1$ is a polynomial in $x$ and $y^2$.
In particular, $|g(re^{i\theta})|^2 - 1$ is a polynomial in $r\cos\theta$ and $r^2\sin^2\theta$, which can also be written as a polynomial in $r$ and $\cos\theta$.
This observation motivates the following definition.

\begin{dfn}\label{dfn: Gi}
    We define two polynomials $G_1(r, T)$ and $G_2(x, Y)$ by
    \begin{align}
        &|g(re^{i\theta})|^2 = G_1(r, \cos \theta) + 1,\text{ and}\\
        &|g(x + yi)|^2 = G_2(x, y^2) + 1.
    \end{align}
    Explicitly, these polynomials are given by
    \begin{align}
        G_1(r, T) ={} & 8r^3 T^3 + (28 r^4 - 4r^2) T^2 + (-2r^5 - 20 r^3 + 14 r)T + (r^6 - 13 r^4 + 51 r^2),\\
        G_2(x,Y) ={} & Y^3 + (3x^2 - 2x - 13)Y^2 + (3x^4 - 4x^3 + 2x^2 - 20x + 51)Y\\
        &+ (x^6 - 2x^5 + 15x^4 - 12x^3 + 47x^2 + 14x).
    \end{align}
\end{dfn}

\begin{lem} \label{lem: G_has_no_roots}
    For sufficiently large $R > 0$ and sufficiently small $\varepsilon > 0$, the polynomial $G_1(r, T)$ has no roots in
    \begin{align}
        S_1 &= \left\{
            (r, 0) \in \R^2 \mid r > 0
        \right\},\\
        S_2 &= \left\{
            (R, T) \in \R^2 \mid -1 \leq T \leq 1
        \right\}, \\
        S_3 &= \left\{
            (s_1, T) \in \R^2 \mid -1 \leq T \leq 0
        \right\},\\
        S_4 &= \left\{
            (s_2, T) \in \R^2 \mid 0 \leq T \leq 1
        \right\}, \text{ and}\\
        S_5 &= \left\{
            (\varepsilon, T) \in \R^2 \mid 0 \leq T \leq 1
        \right\}.
    \end{align}

    Therefore, $\Gamma_k(t)$ has no roots in the regions in the complex plane displayed in \cref{fig: roots_range1}.
\end{lem}

\begin{figure}[htb] 
    \includegraphics[width=8cm]{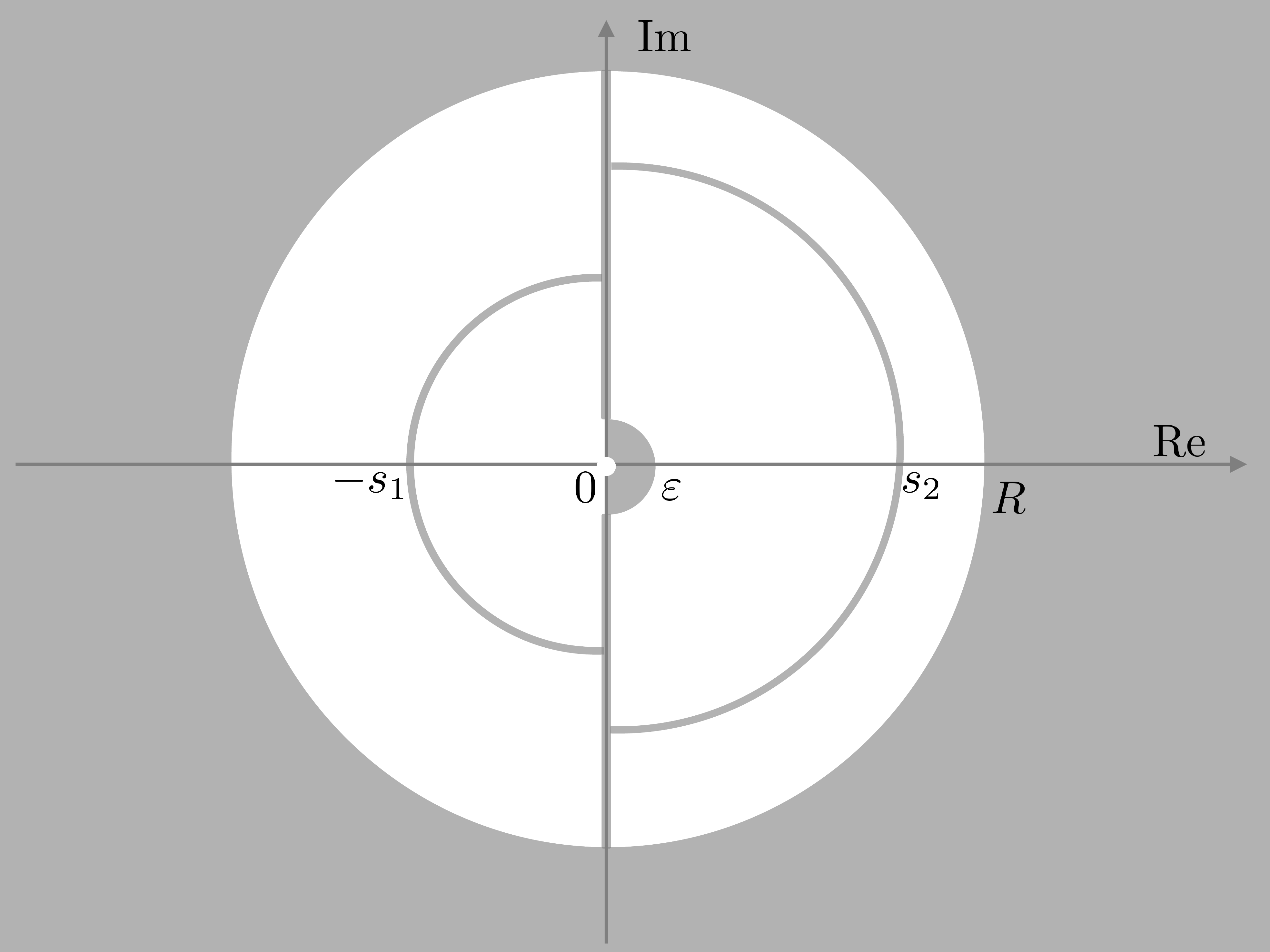}
    \caption{Regions where $\Gamma_k(t)$ has no roots}
    \label{fig: roots_range1}
\end{figure}

\begin{proof}
    Since we have
    \[
        G_1(r,0) = r^2\left((r^2 - 13/2)^2 + 35/4\right) > 0
    \]
    for $r > 0$, the polynomial $G_1(r, T)$ has no roots in $S_1$.
    The value $\min\left\{ G_1(r, T) \mid -1\leq T \leq 1\right\}$ is bounded from below by a monic polynomial in $r$ of degree $6$, so it diverges to $+\infty$ as $r \to \infty$.
    Hence, the polynomial $G_1(r, T)$ has no roots in $S_2$ for sufficiently large $R$.

    To prove the statements for $S_3$, $S_4$, and $S_5$, set
    \[
        T_{\pm}(r) = \frac{-7r^2 + 1 \pm \sqrt{52r^4 + 16 r^2 - 20}}{6r},
    \]
    which are the critical points of redthe function $T\mapsto G_1(r, T)$.

    \begin{itemize}
    \item   Since $T_-(s_1)$ and $T_+(s_1)$ are not real, the value $G_1(s_1, T)$ is monotonely increasing in $T$. Thus, we have the inequalities
    \[
        \min\left\{ G_1(s_1, T) \mid -1\leq T\leq 0\right\} = G_1(s_1, -1) = 0.04330\cdots > 0.
    \]
    Hence, the polynomial $G_1(r, T)$ has no roots in $S_3$.

    \item   Since we have $T_-(s_2) < 0 < T_+(s_2) < 1$,
    we get the inequalities
    
    \[
        \min\left\{ G_1(s_2, T) \mid 0\leq T\leq 1\right\} = G_1(s_2, T_+(s_2)) = 0.2027\cdots > 0.
    \]
    
    Hence, the polynomial $G_1(r, T)$ has no roots in $S_4$.

    \item   Since $T_\pm(\varepsilon)$ is not real for sufficiently small $\varepsilon$,
    the value $G(\varepsilon, T)$ is monotonely increasing in $T$.
    Thus, we have the inequalities
    \[
        \min\left\{ G_1(\varepsilon, T) \mid 0\leq T\leq 1\right\} = G(\varepsilon, 0) = \varepsilon^2\left((\varepsilon^2 - 13/2)^2 + 35/4\right) > 0.
    \]
    Hence, the polynomial $G_1(r, T)$ has no roots in $S_5$.
    \end{itemize}
\end{proof}

\begin{lem}\label{lem: distribution_of_zeros_on_stripes}
    The polynomial $G_2(x, Y)$ has no roots in
    \begin{align}
        S_6 &= \left\{ (x, Y) \in \mathbb{R}^2 \mid x = s_3,\ Y \geq 0 \right\}, \text{ and}\\
        S_7 &= \left\{ (x, Y) \in \mathbb{R}^2 \mid x = s_4,\ Y \geq 0 \right\}.
    \end{align}
    Therefore, $\Gamma_k(t)$ has no roots in two lines in the complex plane displayed in \cref{fig: roots_range2}.
\end{lem}

\begin{figure}[htb] 
    \includegraphics[width=8cm]{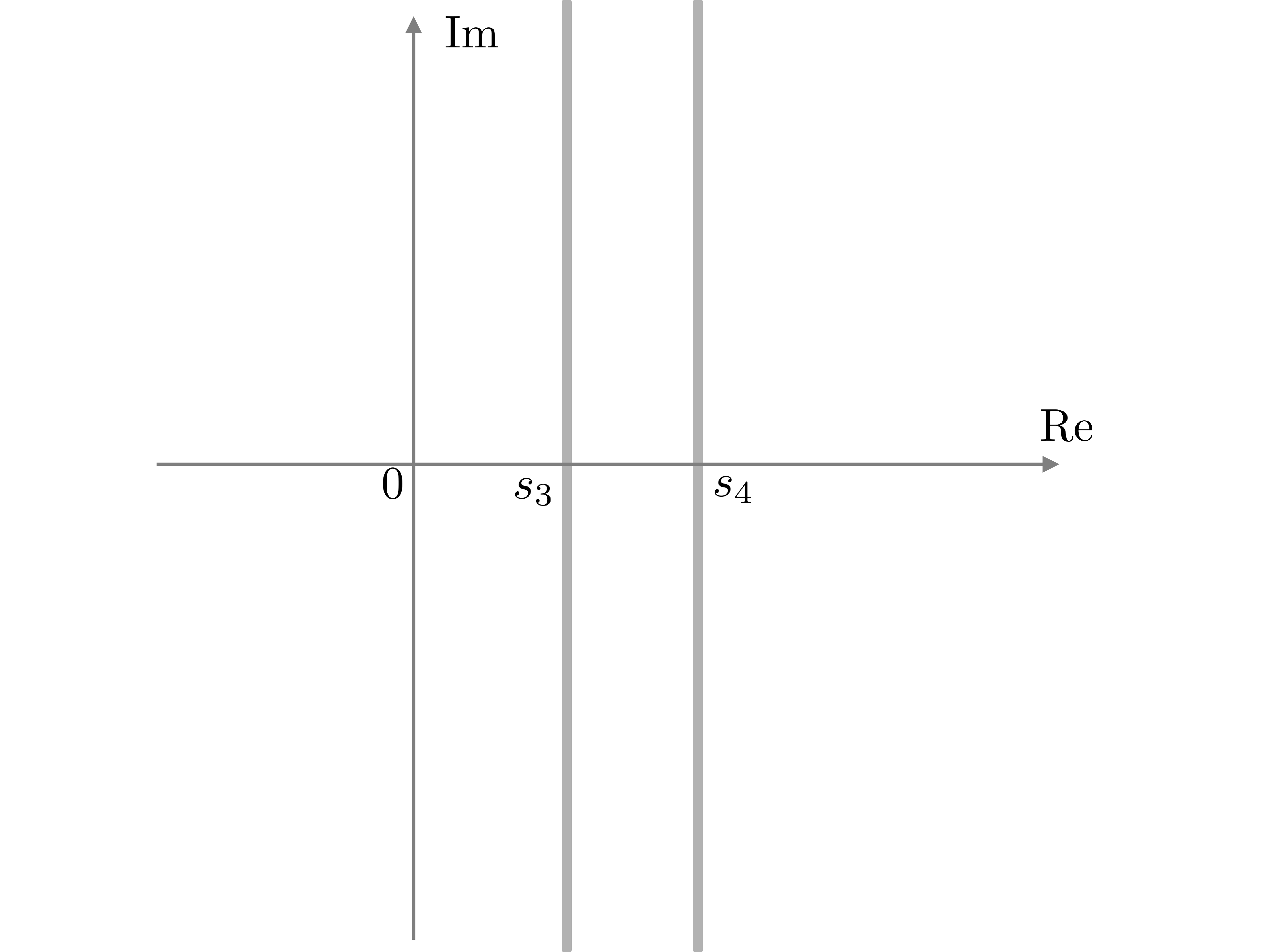}
    \caption{Two lines where $\Gamma_k(t)$ has no roots}
    \label{fig: roots_range2}
\end{figure}

\begin{proof}
    For each $x\in \R$, we set
    \[
        Y_\pm(x) = \frac{-3x^2 + 2x + 13 \pm \sqrt{-80x^2 + 112x + 16}}{3},
    \]
    which are the critical points of $G_1(x, Y)$ as a function of $Y$.
    Since we have $Y_+(s_3) > 0$ and $Y_+(s_4) > 0$, we get the inequalities
    \begin{align}
        \phantom{=}\min\{G_2(s_3, Y)\mid Y\geq 0\}
        &= \min\{G_2(s_3, 0), G_2(s_3, Y_+(s_3))\}\\
        &= \min\{17.8465\ldots, 0.1065\ldots\} > 0, \text{ and}\\
        \phantom{=}\min\{G_2(s_4, Y)\mid Y\geq 0\}
        &= \min\{G_2(s_4, 0), G_2(s_4, Y_+(s_4))\}\\
        &= \min\{27.436\ldots, 0.00023\ldots\} > 0.
    \end{align}
    The assertion follows from these inequalities.
\end{proof}

For a complex number $z_0$ and a positive real number $r$, let $D(z_0,r)$ be the open disc of radius $r$ centered at $z_0$.
Then, the following lemma on the distribution of roots of $g(z) - z_0$ for $z_0 \in \bbC$ with $|z_0| = 1$ holds.
\begin{lem}\label{lem: distributions_of_zeros_on_circles}
We have the following statements.
    \begin{enumerate}
        \item\label{item:lem: A_is_contained}
            For all $z_0 \in \bbC$ with $|z_0| = 1$ and sufficiently small $\varepsilon>0$, the polynomial $g(z) - z_0$ has exactly one root in $D_1 \coloneqq (\bbC_L \cap D(0, s_1)) \cup D\left(0, \varepsilon\right)$.

        \item\label{item:lem: B_is_contained}
            For all $z_0 \in \bbC$ with $|z_0| = 1$ and sufficiently small $\varepsilon>0$, the polynomial $g(z) - z_0$ has exactly two roots in $D_2 \coloneqq (\bbC_R\cap D\left(0, s_2\right)) \setminus D\left(0, \varepsilon\right)$.

        \item\label{item:lem: no_roots_in_left_hand_side}
            For all $z_0 \in \bbC$ with $|z_0| = 1$, the polynomial $g(z) - z_0$ has no roots in $D_3 \coloneqq \bbC_L \setminus D\left(0, s_1\right)$.

        \item\label{item:lem: no_roots_in_right_hand_side}
            For all $z_0 \in \bbC$ with $|z_0| = 1$, the polynomial $g(z) - z_0$ has no roots in $D_4 \coloneqq \bbC_R \setminus D\left(0, s_2\right)$.
    \end{enumerate}
\end{lem}
\begin{proof}
Putting $z = re^{i\theta}$ with $r \in \R_{\geq 0}$ and $\theta \in [0,2\pi)$,
we have $|g(z)|^2 = G_1(r, \cos\theta) + 1$, where $G_1(r, T)$ is the polynomial defined in \cref{dfn: Gi}.
Consider the paths $\Omega_1(r,\varepsilon)$ and $\Omega_2(r,\varepsilon)$ shown in \cref{fig: integration_paths_1}.

\begin{figure}[htb]
\centering
\begin{tabular}{cc}
\begin{tikzpicture}[
    arrow/.style={
        postaction={
            decorate,
            decoration={
                markings,
                mark=at position #1 with {\arrow{stealth}}
            }
        }
    },
    label2/.style 2 args={
        pos/.style={
            postaction={
                decorate,
                decoration={
                    markings,
                    mark=at position ##1 with \node #2;
                }
            }
        }
    },
    label/.style={
        label2={1}{#1}
    },
    pos/.default=.5,
    arrow/.default=.5,
    scale = 0.8
]

\draw[->] (-4, 0) -- (3, 0) node[right]{$\Re z$};
\draw[->] (0, -4) -- (0, 4) node[above]{$\Im z$};

\draw (0,0) node[below left]{0};

\draw[very thick,domain=-3:-0.6,
    label={[right]{}},
    pos, arrow
    ]
plot(0,{\x});

\draw[very thick,domain= 0.6: 3,
    label={[above]{}}, pos, arrow=.5]
plot (0,{\x});

\draw[very thick,domain=90:270,variable=\t,label={[above left]{$\Omega_1(r,\varepsilon)$}},
    pos={.25}, arrow=.75]
    plot({3*cos(\t)}, {3*sin(\t)});

\draw[very thick,domain=-90:90,variable=\t,label={[right]{}},pos={.25},arrow=.5]  plot ({0.6*cos(\t)},{0.6*sin(\t)});

\draw (-3, 0) node[above right]{$z=-r$};
\draw (0, -.6) node[below right]{$z=-\varepsilon i$};
\draw (0, .6) node[above right]{$z=\varepsilon i$};
\draw (.6, 0) node[below right]{$z=\varepsilon$};
\draw (0,  3) node[right]{$z=ri$};
\draw (0,  -3) node[right]{$z=-ri$};
\end{tikzpicture}

&

\begin{tikzpicture}[
    arrow/.style={
        postaction={
            decorate,
            decoration={
                markings,
                mark=at position #1 with {\arrow{stealth}}
            }
        }
    },
    label2/.style 2 args={
        pos/.style={
            postaction={
                decorate,
                decoration={
                    markings,
                    mark=at position ##1 with \node #2;
                }
            }
        }
    },
    label/.style={
        label2={1}{#1}
    },
    pos/.default=.5,
    arrow/.default=.5,
    scale = 0.8 
]

\draw[->] (-2.5, 0) -- (4, 0) node[right]{$\Re z$};
\draw[->] (0, -4) -- (0, 4) node[above]{$\Im z$};

\draw (0,0) node[below left]{0};

\draw[very thick,domain=-.6: -3,
    label={[right]{}},
    pos, arrow
    ]
plot(0,{\x});

\draw[very thick, domain= 3: .6,
    label={[above]{}}, pos, arrow=.5]
plot (0, {\x});

\draw[very thick,domain=-90:90,variable=\t,label={[below right]{$\Omega_2(r,\varepsilon)$}},
    pos={.25}, arrow=.75]
    plot({3*cos(\t)}, {3*sin(\t)});

\draw[very thick, domain=90:-90,variable=\t,label={[right]{}}, pos={.25},arrow=.5]  plot ({0.6*cos(\t)},{0.6*sin(\t)});

\draw (3, 0) node[above left]{$z=r$};
\draw (0, -.6) node[below left]{$z=-\varepsilon i$};
\draw (0, .6) node[above left]{$z=\varepsilon i$};
\draw (.6, 0) node[below right]{$z=\varepsilon$};
\draw (0,  3) node[left]{$z=ri$};
\draw (0, -3) node[left]{$z=-ri$};
\end{tikzpicture}
\end{tabular}
\caption{Integration paths $\Omega_1(r,\varepsilon)$ and $\Omega_2(r,\varepsilon)$}
\label{fig: integration_paths_1}
\end{figure}

By \cref{lem: G_has_no_roots}, if $z_0 \in \bbC$ satisfies $|z_0|=1$,
the polynomial $g(z) - z_0$ has no roots and poles on $\Omega_1(s_1, \varepsilon)$, $\Omega_1(R, \varepsilon)$, $\Omega_2(s_2, \varepsilon)$, and $\Omega_2(R, \varepsilon)$ for sufficiently large $R > 0$ and sufficiently small $\varepsilon > 0$.
For such $R$ and $\varepsilon$, the number of zeros of $g(z) - z_0$ in $D_1, D_2, D_3$, and $D_4$ are equal to the integrations
\begin{align}
    I_1(z_0) &\coloneqq \dfrac{1}{2 \pi i}\oint_{\Omega_1(s_1, \varepsilon)} \frac{g'(z)}{g(z) - z_0} dz,\\
    I_2(z_0) &\coloneqq \dfrac{1}{2\pi i}\oint_{\Omega_2(s_2, \varepsilon)} \frac{g'(z)}{g(z) - z_0} dz,\\
    I_3(z_0) &\coloneqq \dfrac{1}{2\pi i}\oint_{\Omega_1(R, \varepsilon)} \frac{g'(z)}{g(z) - z_0} dz - I_1, \text{ and }\\
    I_4(z_0) &\coloneqq \dfrac{1}{2\pi i}\oint_{\Omega_2(R, \varepsilon)} \frac{g'(z)}{g(z) - z_0} dz - I_2,
\end{align}
respectively, by the argument principle.
The quantities $I_j(z_0)\;(1\leq j \leq 4)$ are well-defined and continuous in $z_0$.
On the other hand, since $I_j(z_0)$ are integers, they are constant as a function of $z_0$.
Therefore, it suffices to calculate the values $I_j(1)$ $(1\leq j \leq 4)$.
Since we have
\[
    g(z) - 1 = z^3 - z^2 + 7z = z\left(z - \frac{1 + 3\sqrt{-3}}{2}\right)\left(z - \frac{1 - 3\sqrt{-3}}{2}\right),
\]
we get $I_1(1) = 1$, $I_2(1) = 2$, and $I_3(1) = I_4(1) = 0$.
\end{proof}

\begin{lem}\label{lem: A_B_real_part}
    We have $\Re(A)\subset [-s_1, 0]$ and $\Re(B) \subset [s_3, s_4]$.
    Moreover, $z=0$ is the only value such that $z\in A$ and $\Re(z)=0$.
\end{lem}

\begin{proof}
    The inclusion $\Re(A)\subset [-s_1, 0]$ follows from \cref{lem: distributions_of_zeros_on_circles}.
    
    Putting $z = x + y\sqrt{-1}$ with $x,y \in \R$, we have $|g(z)|^2 - 1 = G_2(x, y^2)$.
    Consider the path $\Omega_3(y_0)$ shown in \cref{fig: integration_path_2}.
\begin{figure}[h]
\begin{tikzpicture}[
    arrow/.style={
        postaction={
            decorate,
            decoration={
                markings,
                mark=at position #1 with {\arrow{stealth}}
            }
        }
    },
    label2/.style 2 args={
        pos/.style={
            postaction={
                decorate,
                decoration={
                    markings,
                    mark=at position ##1 with \node #2;
                }
            }
        }
    },
    label/.style={
        label2={1}{#1}
    },
    pos/.default=.5,
    arrow/.default=.5,
    scale = 0.8
]

\draw[->] (-2,0) -- (4,0) node[right]{$\Re z$};

\draw (0,0) node[below right]{0};

\draw[very thick,domain=3: -3,
    label={[right]{}},
    pos, arrow
    ]
plot(0.8,{\x});

\draw[very thick,domain= -3: 3,
    label={[above right]{}}, pos, arrow=.5]
plot (1.5,{\x});

\draw[very thick,domain= 1.5: 0.8,
    label={[right]{}}, pos, arrow=.5]
plot ({\x}, 3);

\draw[very thick,domain= 0.8: 1.5,
    label={[above]{}}, pos, arrow=.5]
plot ({\x}, -3);

\draw (1.15, 3) node[above]{$y = y_0$};
\draw (1.15, -3) node[below]{$y = -y_0$};
\draw (0.8, 0.5) node[left]{$x = s_3$};
\draw (1.5, 0.5) node[right]{$x = s_4$};
\draw (1.5, 1.5) node[right]{$\Omega_3(y_0)$};
\end{tikzpicture}
\caption{Integration path $\Omega_3(y_0)$}
\label{fig: integration_path_2}
\end{figure}
    Suppose that $z_0\in \bbC$ satisfies $|z_0| = 1$.
    For sufficiently large $y_0 \in \R$, the polynomial $g(z) - z_0$ has no roots and poles on $\Omega_3(y_0)$ by \cref{lem: distribution_of_zeros_on_stripes} and \cref{lem: distributions_of_zeros_on_circles}.
    Considering the integration
    \[
        I_5(y_0) \coloneqq \dfrac{1}{2\pi i}\oint_{\Omega_3(y_0)} \frac{g'(z)}{g(z) - z_0} dz
    \]
    for such $y_0$, an argument similar to that of \cref{lem: distributions_of_zeros_on_circles} shows that $g(z) - z_0$ has exactly two roots with $\Re(z) \in [s_3,s_4]$.
    This observation proves the first claim.
    The second claim follows from the fact that the polynomial
    $G_2(0,Y) = Y^3 - 13Y^2 + 51Y$
    has no real roots other than $Y=0$.
\end{proof}

The following two propositions follow from \cref{lem: distributions_of_zeros_on_circles}.
\begin{prop}\label{prop: A_B_absolute_value}
    We have the inclusions $A \subset D(0, s_1)$ and $B \subset D(0, s_2)$.
\end{prop}

\begin{prop}\label{prop: number_of_roots_in_A}
    The polynomial $\Gamma_k$ has exacly $\varphi(k)$ roots in $A$ and exactly $2\varphi(k)$ roots in $B$.
\end{prop}


\section{Irreducibility of the delta factors for period $3$} \label{sec: irreducibility}


We continue to write $s_1 = 0.275$, $s_2 = 2.75$, $s_3 = 0.495$, and $s_4 = 0.64$.

\begin{lem}\label{lem: number of roots in A}
    Let $k\geq 2$ be an integer.
    Let $\alpha$ be a root of $\Gamma_k$ and write $n$ for the number of conjugates of $\alpha$ contained in $A$.
    Then, we have
    \[
        \dfrac{n}{\varphi(k)}\geq - \log_{10}(s_1s_2) = 0.1213\cdots.
    \]
\end{lem}

\begin{proof}
    Let $f_1$ be the minimal polynomial of $\alpha$ and write $\Gamma_k = f_1 f_2$.
    Set $d = \deg f_1$.
    By \cref{lem: parabolic parameter involves cyclotomic field} and \cref{lem: gamma_vs_delta}, we have $d \geq \varphi(k)$.
    By \cref{prop: number_of_roots_in_A}, the number of roots of $f_1$ and $f_2$ contained in $A$ and $B$ are given as follows:
    \begin{table}[h]
        \begin{tabular}{c|ccc}
                  & $A$ & & $B$\\ \hline
            $f_1$ & $n$ & & $d-n$\\
            $f_2$ & $\varphi(k)-n$ & & $2\varphi(k)-d+n$
        \end{tabular}
    \end{table}
    
    Using \cref{prop: A_B_absolute_value} and $d \geq \varphi(k)$, we get the inequalities
    \[
        1 \leq |f_2(0)|
        \leq s_1^{\varphi(k)-n} s_2^{2\varphi(k) - d + n}
        \leq s_1^{\varphi(k)-n} s_2^{\varphi(k) + n}
        = (s_1s_2)^{\varphi(k)} \cdot 10^n.
    \]
    By taking the logarithm of both sides, we get the desired inequality.
\end{proof}

\begin{lem}\label{lem: small totally real integers}
    Let $\alpha$ be a totally real algebraic integer.
    If all conjugates of $\alpha$ are contained in $(-\sqrt{2}, \sqrt{2})$, then we have $\alpha \in \{-1,0,1\}$.
\end{lem}

\begin{proof}
    Let $\beta = \alpha^2-1$.
    Our assumption implies that all conjugates of $\beta$ are contained in $[-1, 1)$.
    If $\alpha \neq 0$, then we have $\beta \neq -1$, so the product of all conjugates of $\beta$ is contained in $(-1,1) \cap \mathbb{Z} = \{0\}$.
    This implies $\beta = 0$ and hence $\alpha \in \{-1, 1\}$.
\end{proof}

Let $\mathbb{Q}^\ab$ denote the maximal abelian extension of $\mathbb{Q}$.

\begin{thm}\label{thm: delta 3 has no abelian roots}
    For any $k\geq 2$, the polynomial $\widetilde{\Delta}_{3k,3}$ has no roots in $\mathbb{Q}^\ab$.
\end{thm}

\begin{proof}
    By \cref{lem: gamma_vs_delta}, there exists a Galois-equivariant bijection between the set of roots of $\widetilde{\Delta}_{3k,3}$ and that of $\Gamma_k$.
    Since an algebraic number is contained in $\mathbb{Q}^\ab$ if and only if it is fixed by the action of the group $\Gal(\overline{\mathbb{Q}}/\mathbb{Q}^\ab)$, it suffices to show that the polynomial $\Gamma_k$ has no roots in $\mathbb{Q}^\ab$.
    Suppose that $\alpha \in \mathbb{Q}^\ab$ is a root of $\Gamma_k$.
    Then, for any $\sigma \in \Gal(\mathbb{Q}^\ab/\mathbb{Q})$, we have
    \[
        \sigma(\alpha + \overline{\alpha})
        = \sigma(\alpha) + \overline{\sigma(\alpha)}
        = 2\Re(\sigma(\alpha)).
    \]
    In particular, the number $\beta = \alpha + \overline{\alpha}$ is a totally real algebraic integer.
    By \cref{lem: number of roots in A}, there is some $\sigma \in \Gal(\mathbb{Q}^\ab/\mathbb{Q})$ such that $\sigma(\alpha) \in A$.
    By \cref{lem: A_B_real_part}, such $\sigma$ satisfies
    \begin{align}\label{eq: beta estimate 1}
        \sigma(\beta) = 2\Re(\sigma(\alpha)) \in [-2s_1, 0),
    \end{align}
    so it follows that $\beta\not\in \{-1, 0, 1\}$.
    On the other hand, for any $\sigma \in \Gal(\mathbb{Q}^\ab/\mathbb{Q})$ with $\sigma(\alpha) \in B$, we have
    \begin{align}\label{eq: beta estimate 2}
        \sigma(\beta) = 2\Re(\sigma(\alpha))\in [2s_3, 2s_4]
    \end{align}
    by \cref{lem: A_B_real_part}.
    In particular, all conjugates of $\beta$ are contained in $(-\sqrt{2}, \sqrt{2})$.
    By \cref{lem: small totally real integers}, this leads to a contradiction.
\end{proof}

\begin{rem}
    In the proof of \cref{thm: delta 3 has no abelian roots}, instead of using \cref{lem: small totally real integers}, we can also use a general theorem concerning the height of algebraic integers.
    Let $d$ be the degree of $\beta$ over $\mathbb{Q}$ and $n$ be the number of conjugates of $\beta$ contained in $A$.
    Using the estimates \cref{eq: beta estimate 1}, \cref{eq: beta estimate 2}, and \cref{lem: number of roots in A}, the absolute logarithmic height $h(\beta)$ of $\beta$ can be bounded as follows:
    \begin{align}
        h(\beta) \leq \dfrac{1}{d}(d - n) \log (2s_4)
        \leq \biggl(1-\dfrac{-\log_{10}(s_1s_2)}{3}\biggr)\log (2s_4)
        = 0.2368\cdots.
    \end{align}
    This inequality contradicts the following theorem.
\end{rem}
\begin{thm}[Schinzel \cite{Schinzel73}]
    Let $\alpha$ be a totally real algebraic integer.
    If $\alpha\not\in \{-1, 0, 1\}$, then we have the inequality
    \[
        h(\alpha)\geq \dfrac{1}{2}\log\dfrac{1 + \sqrt{5}}{2} = 0.2406\cdots.
    \]
\end{thm}


\begin{proof}[Proof of $ \cref{thm: irreducibility for period 3} $]
    Since we have $\widetilde{\Delta}_{3,3} = C+7$, we may assume that $k\geq 2$.
    Suppose that we have $\widetilde{\Delta}_{3k,3}=f_1 f_2$ for some monic polynomials $f_1, f_2\in \mathbb{Q}(\zeta_{\ell})[t]$ with positive degrees.
    For any root $\alpha$ of $\widetilde{\Delta}_{3k,3}$, the field $\mathbb{Q}(\zeta_\ell)(\alpha)$ contains $\mathbb{Q}(\zeta_\ell)(\zeta_k) = \mathbb{Q}(\zeta_{k\ell})$ by \cref{lem: parabolic parameter involves cyclotomic field}. Hence, the extension degree $[\Q(\zeta_{\ell})(\alpha):\Q(\zeta_\ell)]$ is divided by $[\Q(\zeta_{\ell})(\zeta_k):\Q(\zeta_{\ell})] = \varphi(k)$, where note that we are assuming $(k,\ell) = 1$.
    In particular, $\varphi(k)$ divides both $\deg f_1$ and $\deg f_2$.
    Since we have $\deg f_1 + \deg f_2 = \deg \widetilde{\Delta}_{3k,3} = 3\varphi(k)$, we may assume that $\deg f_1 = \varphi(k)$ and $\deg f_2 = 2\varphi(k)$ hold.
    This equality implies that we have $\Q(\zeta_{\ell})(\alpha') = \Q(\zeta_{\ell})(\zeta_k) = \Q(\zeta_{k\ell})$ for a root $\alpha'$ of $f_1$.
    However, this contradicts \cref{thm: delta 3 has no abelian roots}.
\end{proof}



\section{Higher periods} \label{sec: higher periods}


In this section, we prove \cref{thm: irreducibility inherited}.
Throughout this section, we fix integers $ e, m \ge 1 $ and $ k \ge 2 $ and a prime number $ p $ satisfying $ p \nmid k $.

In our proof, we demonstrate that any irreducible component of $ \widetilde{\Delta}_{mkp^e, m}(C) $ is equal to $ \widetilde{\Delta}_{mkp^e, m}(C) $ by counting the number of its roots in two different ways. 
To achieve this, we prepare the following two lemmata.

\begin{lem}\label{lem: mod_p_decomposition_Delta}
	We have the congruence relation
	\[ 
        \widetilde{\Delta}_{mkp^e, m}(C)
        \equiv\widetilde{\Delta}_{mk, m}(C)^{\varphi(p^e)} \bmod p.
	\]
\end{lem}

\begin{proof}
    Since $ \cyc_{kp^e}(x) = \cyc_{k}(x^{p^e}) / \cyc_k(x^{p^{e-1}}) \equiv \cyc_{k}(x)^{\varphi(p^e)} \bmod p $, we have
    \begin{align}
        \widetilde{\Delta}_{mkp^e, m}(C) &= \Resultant_x (\cyc_{kp^e}(x), \widetilde{\delta}_{m}(x, C)) \\
        & \equiv \Resultant_x (\cyc_{k}(x)^{\varphi(p^e)}, \widetilde{\delta}_{m}(x, C)) \bmod p\\
        & = \Resultant_x (\cyc_{k}(x), \widetilde{\delta}_{m}(x, C))^{\varphi(p^e)} \\
        & = \widetilde{\Delta}_{mk,m}(C)^{\varphi(p^e)}.
    \end{align}
\end{proof}

\begin{lem}\label{lem: mod_p_decomposition_irr_comp}
	For a monic irreducible factor $ f(C) $ of $ \widetilde{\Delta}_{mkp^e, m} (C) $ over $ \Q $, there exists a polynomial $ F(C) \in \overline{\bbF}_p [C] $ such that $ f(C) \equiv F(C)^{\varphi(p^e)} \bmod p $.
\end{lem}

\begin{proof}
	Take any root $ \alpha \in \overline{\Q} $ of $ f $.
	Set $ L \coloneqq \bbQ(\alpha) $ and $ K \coloneqq \bbQ(\zeta_{k p^e}) $.
	Then, $ L $ contains $ K $ by \cref{lem: parabolic parameter involves cyclotomic field}.
	Let $ h \in K[C] $ be the minimal polynomial of $ \alpha $ over $ K $.
	
	To begin with, we prove the equality 
	\[
    	f = \prod_{\sigma \in \Gal(K/\Q)} \sigma(h).
	\]
	Since $ \sigma(h) \mid f $ for any $ \sigma \in \Gal(K/\Q) $, it suffices to show that $ \sigma(h) \neq \tau(h) $ for any $ \sigma \neq \tau $, that is, $ \sigma(h) \neq h $ for any $ \sigma \neq 1 $.
	Suppose that $ \sigma(h) = h $.
	Since $ h \in K^\sigma [C] $ and $ L = K^\sigma (\alpha) $, we have
	$ [L \colon K^\sigma]  \le \deg h = [L \colon K] $.
	Thus, we have $ K^\sigma = K $.
	By Galois theory, we have $ \sigma = 1 $.
	
	Set $ K' \coloneqq \Q(\zeta_k) $ and take a prime ideal $ \frakp $ of $ \calO_{K'} $ above $ p $.
	By the assumption $ p \nmid k $, the prime ideal $ \frakp $ ramifies completely in $ \calO_K $, that is, $ \frakp \calO_K = \frakP^{\varphi(p^e)} $ for some prime ideal $ \frakP $ of $ \calO_{K} $.
	Then, for any $ \sigma \in \Gal(K/\Q) $ and $ \tau \in \Gal(K/K') $, we have $ \sigma \tau (h) \equiv \sigma(h) \bmod \frakP $.
	Therefore, we have
	\[
	   f \equiv \prod_{\sigma \in \Gal(K/K')} \sigma(h)^{\varphi(p^e)} \bmod \frakP.
	\]
	By letting $ F $ be the image of $ \prod_{\sigma \in \Gal(K/K')} \sigma(h) $ under $ \calO_K / \frakP \hookrightarrow \overline{\bbF}_p $, we obtain the claim.
\end{proof}

Finally, we prove \cref{thm: irreducibility inherited}.

\begin{proof}[Proof of $ \cref{thm: irreducibility inherited} $]
	Assume that $ \widetilde{\Delta}_{mk, m} $ is irreducible over $ \bbF_p $.
	Then, the congruence relation
	\[ 
	   \widetilde{\Delta}_{mkp^e, m} \equiv \widetilde{\Delta}_{mk, m}^{\varphi(p^e)} \bmod p
	\]
	in \cref{lem: mod_p_decomposition_Delta} is the irreducible decomposition of $ \widetilde{\Delta}_{mkp^e, m} $ over $ \bbF_p $.
	Take a monic irreducible factor $ f $ of $ \widetilde{\Delta}_{mkp^e, m} $ over $ \Q $.
	Then, we have
	$f \equiv \widetilde{\Delta}_{mk, m}^{N} \bmod p$
	for some positive integer $ N $.
	Thus, we have
	\[
    	\# \{ \text{roots of $ f $ in $ \overline{\bbF}_p $} \}
    	= 
    	\# \{ \text{roots of $ \widetilde{\Delta}_{mk, m} $ in $ \overline{\bbF}_p $} \}
    	= 
    	\deg \widetilde{\Delta}_{mk, m}
    	=
    	\varphi(k) \cdot \deg_c \delta_m (x, c).
	\]
	On the other hand, by \cref{lem: mod_p_decomposition_irr_comp}, there exists a polynomial $ F(C) \in \overline{\bbF}_p [C] $ such that $ f(C) \equiv F(C)^{\varphi(p^e)} \bmod p $.
	Thus, we have
	\[
    	\# \{ \text{roots of $ f $ in $ \overline{\bbF}_p $} \}
    	= 
    	\# \{ \text{roots of $ F $ in $ \overline{\bbF}_p $} \}
    	\le
    	\deg F
    	=
    	\frac{\deg f}{\varphi(p^e)}.
	\]
	By combining these two evaluations, we obtain the inequality
	\[
    	\varphi(k) \cdot \deg_c \delta_m (x, c)
    	\le
    	\frac{\deg f}{\varphi(p^e)},
	\]
	which is equivalent to the inequality
	\[
    	\deg f
    	\ge
    	\varphi(k p^e) \cdot \deg_c \delta_m (x, c)
    	=
    	\deg \widetilde{\Delta}_{mk p^e, m}.
	\]
	Thus, we have $ f = \widetilde{\Delta}_{mk p^e, m} $.
\end{proof}

\begin{ex} 
In \cref{tab: irreducibility inherited}, we give examples satisfying the assumption in \cref{thm: irreducibility inherited}, computed using SageMath.

{\renewcommand{\arraystretch}{1.2}
\begin{table}[hptb]
	\centering
        \caption{Tuples $ (m, k, p) $ such that $ \widetilde{\Delta}_{mk, m} (C) $ is irreducible over $ \bbF_p $.}
	\label{tab: irreducibility inherited}
	\begin{tabular}{c c l}
		$ m $ & $\ k \ $ & $ p < 1000 $ \\
		\specialrule{.1em}{.05em}{.05em} 
		4   & 2 & 11, 37, 71, 83, 101, 103, 109, 137, 223, 233, 283, 353, 419, 433, 439,\vspace{-.4em} 
        \\
            &   & 449, 479, 509, 541, 547, 587, 739, 797, 811, 827, 857, 887, 953
            \vspace{.2em} 
        \\
    		& 3 & 59, 167, 239, 419, 449, 617, 683, 701, 719, 743, 863 \\
    		& 4 & 3, 31, 139, 151, 331, 479, 607, 743, 839, 883 \\
    		& 5 & 43, 103, 233, 317, 463, 503, 523, 547, 587, 983, 997 \\
    		& 6 & 137, 233, 251, 491, 521, 647, 821, 929, 947 \\
    		& 7 & 47, 73, 131, 227, 229, 269, 397, 523, 577, 733, 773, 857, 859, 887, 971 \\
    		& 9 & 23, 29, 83, 137, 239, 599, 911 \\
    		& 10 & 113, 193, 233, 257, 307, 317, 683, 757, 787, 853, 857, 863, 887 \\
    		& 11 & 41, 79, 211, 227, 281, 347, 491, 541, 547, 601, 761 \\
    		& 13 & 149, 197, 271, 331, 383, 431, 709 \\
    		& 14 & 227, 409, 661, 761, 773, 857 \\
    		& 17 & 31, 41, 199, 241, 277, 311, 317, 571, 617, 751, 823, 857 \\
    		& 18 & 5, 257, 383, 389, 587, 599, 743, 839 \\
    		& 19 & 29, 53, 281, 421, 433, 523, 547, 857, 971, 991 \\
		\hline
		5   & 2 & 23, 173, 541, 569, 709, 761 \\
            & 3 & 257, 491, 587, 617, 911, 941, 947, 983 \\
            & 4 & 7, 47, 71, 139, 523, 647 \\
            & 5 & 97, 127, 163, 307, 353, 467, 523, 587, 613, 977 \\
            & 6 & 23, 941, 983 \\
            & 7 & 283, 661 \\
            & 9 & 419, 641 \\
            & 10 & 167, 277, 317, 443, 947 \\
            & 11 & 41, 349, 733, 827, 937 \\
            & 13 & 37, 331, 383, 431, 821 \\
            & 14 & 283, 409, 691, 761 \\
            & 17 & 97, 107, 547, 907, 983 \\
            & 18 & 23, 509, 677 \\
            & 19 & 751, 811, 827 \\
		\hline
        6   & 2 & 29, 71, 107, 499, 743, 809, 911, 947 \\
            & 3 & 113, 359 \\
            & 4 & 19, 167, 499, 743 \\
            & 5 & 23, 107, 227, 787 \\
            & 6 & 293, 857, 953 \\
            & 7 & 19 \\
            & 9 & 509, 857 \\
            & 10 & 263, 317, 937 \\
            & 11 & 29, 677 \\
            & 13 & 431, 661 \\
            & 14 & 691, 733 \\
            & 17 & 113, 283, 709 \\
            & 18 & 131, 653, 929, 941 \\
            & 19 & 97, 281, 401 \\
        \hline
	\end{tabular}
\end{table}
}
\end{ex}

\begin{rem}
If $ \cyc_k (x) $ is reducible over $ \bbF_p $ for all $ p $, so is $ \widetilde{\Delta}_{mk, m} (C) $.
Harrison~\cite{Harrison07} proved that $ \cyc_k (x) $ is reducible over $ \bbF_p $ for all $ p $ if and only if the discriminant of $ \cyc_k (x) $ is a square number in $ \Z $.
Such $ k \le 50 $ include
\[
    k = 8, 12, 15, 16, 20, 21, 24, 28, 30, 32, 33, 35, 36, 39, 40, 42, 44, 45, 48.
\]
Here, we remark that the discriminant of $ \cyc_k (x) $ equals to
\[
    (-1)^{ (\varphi(k) / 2) \cdot \# \{ p \mid k \}} n^{\varphi(k)}
    \prod_{p \mid k} p^{-\varphi(k) / (p-1)}
\]
(see \cite[Page 269]{Ribenboim}).
\end{rem}

\newpage
\bibliographystyle{alpha}
\bibliography{dynamic}

\end{document}